\documentclass[11pt,reqno]{amsart}

\usepackage[utf8]{inputenc}
\usepackage[T1]{fontenc}
\usepackage{amsmath,amssymb,amsthm,amsfonts}
\usepackage{mathtools}
\usepackage{geometry}
\usepackage{enumitem}
\usepackage{microtype}
\usepackage{xcolor}
\usepackage[colorlinks=true,linkcolor=blue,citecolor=red,urlcolor=teal]{hyperref}
\newtheorem{theorem}{Theorem}[section]
\newtheorem{proposition}[theorem]{Proposition}
\newtheorem{lemma}[theorem]{Lemma}
\newtheorem{corollary}[theorem]{Corollary}
\theoremstyle{definition}
\newtheorem{definition}[theorem]{Definition}

\newcommand{\N}{\mathbb{N}}
\newcommand{\dist}{\mathrm d}
\newcommand{\rhok}{\rho_{\kappa}}
\newcommand{\bdry}{\partial_\kappa}
\newcommand{\Hdim}{\dim_{\mathrm H}}
\newcommand{\ldimA}{\underline{\dim}_{\mathrm A}}
\newcommand{\diam}{\operatorname{diam}}
\newcommand{\id}{\operatorname{id}}

\title[Uniform Perfectness and Centers]
{Uniform Perfectness and Centers in Sublinearly Morse Boundaries}
\author{Hyungryul Baik}
\address{Department of Mathematical Sciences, KAIST, 291 Daehak-ro, Yuseong-gu, Daejeon 34141, Republic of Korea}
\email{hrbaik@kaist.ac.kr}
\date{Version 17, June 2026}

\begin{document}

\begin{abstract}
The $\kappa$--Morse boundary was introduced for CAT(0) spaces by Qing and
Rafi and extended to proper geodesic spaces by Qing, Rafi, and Tiozzo.
Motivated by Han and Liu's work on uniformly perfect Morse boundaries, we ask
when uniform perfectness of visual boundary data detects $\kappa$--center
exhaustivity.  Two locally finite trees show that fixed-basepoint uniform
perfectness is insufficient and, when $\kappa$ is unbounded, a
basepoint-independent absolute annular cutoff is not necessary.  For locally
finite trees, center exhaustivity is equivalent to fixed-basepoint uniform
perfectness and $\kappa$--radial accessibility.  Under explicit uniform
visual-data hypotheses, analogous conditions imply center exhaustivity
through a chosen boundary stratum in proper geodesic spaces; a normalized
all-basepoint criterion is also obtained.  Finally, we characterize the
metric transforms $\phi$ for which one distortion function, depending only on
$\phi$, works for every identity $(Z,d)\to(Z,\phi\circ d)$, and analyze the
resulting metrics on rooted $q$--ary trees.
\end{abstract}

\maketitle

%=============================================================================
\section{Introduction}

Visual metrics on Gromov boundaries are a basic link between large-scale and
boundary metric geometry; see, for example, \cite{BonkSchramm00}.
Contracting and Morse boundaries extend this viewpoint beyond globally
hyperbolic spaces.  Charney and Sultan introduced the contracting boundary of
a CAT(0) space \cite{CharneySultan15}, while Cordes introduced the Morse
boundary of a proper geodesic space \cite{Cordes17}; the latter paper has a
subsequent corrigendum \cite{CordesCorr24}.  Related developments include the
connection with stable subsets \cite{CordesHume17}, a metrizable topology on
the contracting boundary \cite{CashenMackay19}, and quasi-M\"obius maps of
Morse boundaries \cite{CharneyCordesMurray19}.

Qing and Rafi introduced the $\kappa$--Morse boundary for CAT(0) spaces and
proved that it is quasi-isometry invariant and metrizable \cite{QRT22}.
Qing, Rafi, and Tiozzo extended the construction to arbitrary proper geodesic
spaces \cite{QRT24}.  Related work includes characterizations of sublinearly
Morse geodesics in CAT(0) spaces \cite{MQZ22} and invariance under suitable
sublinear bilipschitz equivalences \cite{PallierQing24}.  These works provide
the topological and functorial background used here.  In our general-space
results, the visual quasi-ultrametrics and their uniform comparison estimates
are additional hypotheses.  We neither modify the boundary construction nor
claim that these estimates follow from metrizability.

Uniform perfectness is a standard scale-nondegeneracy condition in metric
geometry; see \cite{Heinonen01,JarviVuorinen96}.  In the Morse setting, Han
and Liu introduced a gauge-by-gauge notion of uniform perfectness.  For a
proper geodesic space whose Morse boundary contains at least three points,
they proved that the following are equivalent: (i) uniform perfectness
together with uniformly Morse basedness, (ii) Morse geodesic richness, and
(iii) center exhaustivity \cite[Thm.~1.1]{HanLiu26}.  Once the input gauge is
fixed, their target gauge and annular constants are independent of the
basepoint \cite[Def.~3.4]{HanLiu26}.  They also characterize boundary
homeomorphisms induced by quasi-isometries under the equivalent hypotheses
\cite[Thm.~1.4]{HanLiu26}.

We investigate only the center-exhaustivity part of a possible analogue for
$\bdry X$.  The direct analogue fails: a fixed-basepoint condition does not
control finite dead ends, while, when $\kappa$ is unbounded, a single positive
absolute cutoff at all basepoints is not necessary for $\kappa$--center
exhaustivity.  The criteria below use radial accessibility or a
basepoint-dependent cutoff at the $\kappa$--scale.  We do not claim a full
sublinear analogue of the Han--Liu equivalence or its rigidity theorem.

Fix an increasing, concave, sublinear function
$\kappa\colon[0,\infty)\to[1,\infty)$.  Write $|x|=\dist(o,x)$ and
\[
 \rhok(t)=\int_0^t\frac{ds}{\kappa(s)}.
\]
Lemma~\ref{lem:rho-two-sided} gives, for large arguments, the two
conversions between $|s-t|=O(\kappa(t))$ and
$|\rhok(s)-\rhok(t)|=O(1)$.  We consider visual quantities comparable to
\[
 \exp\!\bigl(-\epsilon\rhok(\dist(o,\alpha_{\xi\eta}))\bigr),
\]
where $\alpha_{\xi\eta}$ joins two boundary points.

Theorem~\ref{thm:counterexample} gives a locally finite tree whose geodesics
have a common Morse gauge and whose fixed-basepoint boundary is uniformly
perfect, but which is neither $\kappa$--radially accessible nor
$\kappa$--center--exhaustive.  For locally finite trees,
Theorem~\ref{thm:radial-equivalence} proves that center exhaustivity is
equivalent to radial accessibility together with fixed-basepoint uniform
perfectness.  When $\kappa$ is unbounded,
Theorem~\ref{thm:allbasepoint-obstruction} gives a center-exhaustive tree whose
boundary diameters tend to zero along a sequence of basepoints; thus one
positive absolute annular cutoff cannot be imposed at every basepoint.
Proposition~\ref{prop:normalized-examples} compares the normalized cutoff on
the two examples.

For proper geodesic spaces, Theorem~\ref{thm:general-radial} proves that
fixed-basepoint uniform perfectness together with $\kappa$--radial
accessibility implies $\kappa$--center exhaustivity through $\Lambda$, under
the uniform visual data of Definition~\ref{def:uniform-visual-stratum}.
Proposition~\ref{prop:normalized-allbasepoint} gives a second sufficient
condition under the all-basepoint data of
Definition~\ref{def:all-basepoint-visual}, using the cutoff
\[
 r_{\kappa,A,\lambda}(x)
 =\lambda\exp\!\bigl(-\epsilon\rhok(A\kappa(|x|))\bigr).
\]
These visual-data assumptions are additional to the metrizability established
in \cite{QRT22,QRT24}; they are not asserted to hold on an arbitrary subset of
$\bdry X$.

Earlier versions claimed a fixed-basepoint equivalence and a quasisymmetric
rigidity statement.  Both claims are withdrawn.  The remaining results
concern the obstructions above, the center criteria just stated, and the
metric geometry of the renormalized boundary scale.

For a metric transform $\phi$, put
\[
 \omega_\phi(t)=\sup_{r>0}\frac{\phi(tr)}{\phi(r)}.
\]
Proposition~\ref{prop:relative-scale} proves that there is a quasisymmetric
distortion function depending only on $\phi$ for every identity
\[
 \id:(Z,d)\longrightarrow(Z,\phi\circ d)
\]
if and only if $\omega_\phi(t)\to0$ as $t\downarrow0$.
Proposition~\ref{prop:doubling-counterexample} gives a concave metric transform
for which this limit fails.

For the rooted $q$--ary tree, Theorem~\ref{thm:regular-dichotomy} proves that
the $\rhok$--renormalized boundary metric is uniformly perfect for every
$\kappa$, while quasisymmetry to a standard visual metric and doubling are
equivalent to boundedness of $\kappa$.  It also computes the Hausdorff
dimension.  Section~\ref{sec:dimension} gives a lower Assouad-dimension bound
for bounded uniformly perfect metric spaces.

%=============================================================================
\section{The renormalized scale}\label{sec:scale-prelim}
%=============================================================================

Throughout the paper, all metric balls are closed.  We assume that $\kappa\colon[0,\infty)\to[1,\infty)$ is increasing, concave, and sublinear:
\[
\lim_{t\to\infty}\frac{\kappa(t)}{t}=0.
\]
Bounded functions, including $\kappa\equiv1$, are allowed.

The constants implicit in $\asymp$ and $O(\kappa(R))$ are uniform in the variables quantified in the relevant statement; they may depend on fixed gauges and displayed parameters.  A $Q$--quasi-ultrametric is a symmetric function $d$ that vanishes precisely on the diagonal and satisfies
\[
 d(x,z)\le Q\max\{d(x,y),d(y,z)\}.
\]

\begin{lemma}[Elementary concavity estimates]\label{lem:kappa-concavity}
For $s,t\ge0$ and $\lambda\ge1$,
\[
 \kappa(s+t)\le \kappa(s)+\kappa(t),
 \qquad
 \kappa(\lambda t)\le \lambda\kappa(t).
\]
For $0\le\lambda\le1$,
\[
 \kappa(\lambda t)\ge\lambda\kappa(t).
\]
\end{lemma}

\begin{proof}
Set $g(t)=\kappa(t)-\kappa(0)$.  Then $g$ is nonnegative, increasing, concave, and $g(0)=0$.  The quotient $g(t)/t$ is nonincreasing on $(0,\infty)$.  This gives $g(\lambda t)\le\lambda g(t)$ for $\lambda\ge1$ and $g(\lambda t)\ge\lambda g(t)$ for $0\le\lambda\le1$.  It also gives subadditivity of $g$; for example, assuming $s\le t$,
\[
 g(s+t)\le \frac{s+t}{t}g(t)=g(t)+\frac{s}{t}g(t)\le g(t)+g(s).
\]
For subadditivity, this gives
\[
 \kappa(s+t)=g(s+t)+\kappa(0)
 \le g(s)+g(t)+\kappa(0)
 =\kappa(s)+\kappa(t)-\kappa(0)
 \le\kappa(s)+\kappa(t).
\]
The scaling inequalities for $\kappa$ follow similarly from those for $g$ and the fact that $\kappa(0)\ge0$.
\end{proof}

Define
\[
\rhok(t)=\int_0^t\frac{ds}{\kappa(s)}.
\]
The function $\rhok$ is continuous, strictly increasing, concave, and unbounded.  Indeed, for every $A>0$, eventually $\kappa(s)\le s/A$, and hence $\rhok(t)$ dominates $A\log t$ up to an additive constant.

\begin{lemma}[Sublinear errors become bounded]\label{lem:rho-forward}
For every $A\ge0$ and every $R\ge0$,
\[
0\le \rhok(R+A\kappa(R))-\rhok(R)\le A.
\]
\end{lemma}

\begin{proof}
On the interval $[R,R+A\kappa(R)]$ one has $\kappa(s)\ge\kappa(R)$, and therefore
\[
\rhok(R+A\kappa(R))-\rhok(R)
 =\int_R^{R+A\kappa(R)}\frac{ds}{\kappa(s)}
 \le\frac{A\kappa(R)}{\kappa(R)}=A.
\]
\end{proof}

\begin{lemma}[Two-sided conversion of errors]\label{lem:rho-two-sided}
The following two statements hold.
\begin{enumerate}[label=(\roman*)]
\item For every $A\ge0$ there is $R_A\ge0$ such that, whenever $R\ge R_A$ and
\[
 |s-R|\le A\kappa(R),
\]
one has
\[
 |\rhok(s)-\rhok(R)|\le2A.
\]
\item For every $C\ge0$ there is $R_C\ge1$ such that, whenever $R\ge R_C$ and
\[
 |\rhok(s)-\rhok(R)|\le C,
\]
one has
\[
 |s-R|\le (e-1)C\kappa(R).
\]
\end{enumerate}
\end{lemma}

\begin{proof}
For (i), the case $A=0$ is immediate.  By sublinearity, choose $R_A$ so that
\[
 A\kappa(R)\le \frac R2
\]
for $R\ge R_A$.  If $s\ge R$, then $s\le R+A\kappa(R)$ and Lemma~\ref{lem:rho-forward} gives
\[
 0\le\rhok(s)-\rhok(R)\le A.
\]
If $s\le R$, then $s\ge R/2$.  Lemma~\ref{lem:kappa-concavity} gives
\[
 \kappa(s)\ge\frac{s}{R}\kappa(R)\ge\frac12\kappa(R),
\]
and hence
\[
 0\le\rhok(R)-\rhok(s)
 \le\frac{R-s}{\kappa(s)}
 \le2A.
\]

For (ii), choose $R_C\ge1$ so that $C\kappa(R)/R\le1$ for $R\ge R_C$.  If $s\le R$, then
\[
 C\ge\rhok(R)-\rhok(s)
 \ge\frac{R-s}{\kappa(R)},
\]
so $R-s\le C\kappa(R)$.  If $s\ge R$, Lemma~\ref{lem:kappa-concavity} gives
\[
 \kappa(u)\le\frac{u}{R}\kappa(R)
 \qquad (u\ge R).
\]
Therefore
\[
 C\ge\rhok(s)-\rhok(R)
 \ge\frac{R}{\kappa(R)}\log\frac{s}{R}.
\]
Writing $\delta=C\kappa(R)/R\le1$, we obtain
\[
 s-R\le R(e^\delta-1)
 \le(e-1)R\delta
 =(e-1)C\kappa(R).
\]
Since $e-1>1$, the same displayed bound also covers the case $s\le R$.
\end{proof}

\begin{lemma}[Increment estimates]\label{lem:rho-increments}
For integers $n,L\ge0$ with $L\ge1$,
\[
 \frac{L}{\kappa(n+L)}
 \le \rhok(n+L)-\rhok(n)
 \le \frac{L}{\kappa(n)}.
\]
Consequently:
\begin{enumerate}[label=(\roman*)]
\item $0<\rhok(n+1)-\rhok(n)\le1$;
\item if $\kappa$ is unbounded, then for every fixed $L$,
\[
\rhok(n+L)-\rhok(n)\longrightarrow0;
\]
\item if $\kappa\le M$, then for all $0\le s\le t$,
\[
 \frac{t-s}{M}\le\rhok(t)-\rhok(s)\le t-s.
\]
\end{enumerate}
\end{lemma}

\begin{proof}
Since $\kappa$ is increasing, $1/\kappa$ is decreasing.  Thus on $[n,n+L]$ it lies between $1/\kappa(n+L)$ and $1/\kappa(n)$.  Integrating gives the first display.  Parts (i) and (ii) follow from that display.  If $\kappa\le M$, then $1/M\le1/\kappa(u)\le1$ for all $u$, and integration over $[s,t]$ gives (iii).
\end{proof}

\begin{lemma}[Average renormalized growth]\label{lem:rho-average}
Because $\kappa$ is increasing, the limit $K_\infty:=\lim_{t\to\infty}\kappa(t)$ exists in $[1,\infty]$.  Moreover,
\[
 \lim_{t\to\infty}\frac{\rhok(t)}{t}
 =\begin{cases}
 K_\infty^{-1},&K_\infty<\infty,\\[2mm]
 0,&K_\infty=\infty.
 \end{cases}
\]
\end{lemma}

\begin{proof}
The function $1/\kappa(t)$ converges to $1/K_\infty$, with the convention $1/\infty=0$.  The conclusion is the integral version of the Ces\`aro convergence theorem:
\[
\frac{\rhok(t)}{t}=\frac1t\int_0^t\frac{ds}{\kappa(s)}\longrightarrow \frac1{K_\infty}.
\]
\end{proof}

%=============================================================================
\section{Center criteria}\label{sec:general-radial}
%=============================================================================

Let $(X,\dist,o)$ be a proper geodesic metric space and write
$|x|=\dist(o,x)$.  We use the definitions of $\kappa$--Morse rays and closed
subsets from \cite{QRT24}.  A bi-infinite geodesic is called $\kappa$--Morse
when its image is $\kappa$--Morse as a closed subset.  For common-endpoint
tracking we use \cite[Lem.~3.4(i)]{QRT24}.

\begin{definition}[Uniform $\kappa$--visual data]\label{def:uniform-visual-stratum}
Let $\Lambda\subset\bdry X$ contain at least three points.  A \emph{uniform $\kappa$--visual datum on $\Lambda$ based at $o$} consists of the following.
\begin{enumerate}[label=(\roman*)]
\item For every $\xi\in\Lambda$, a geodesic ray
\[
 \sigma_\xi:[0,\infty)\longrightarrow X,
 \qquad \sigma_\xi(0)=o,
\]
representing $\xi$, such that all the rays $\sigma_\xi$ are $\kappa$--Morse with a common gauge $m_{\mathrm r}$.
\item For every distinct $\xi,\eta\in\Lambda$, a chosen bi-infinite geodesic $\alpha_{\xi\eta}=\alpha_{\eta\xi}$ joining $\xi$ to $\eta$, such that all the closed sets $\alpha_{\xi\eta}$ are $\kappa$--Morse with a common gauge $m_{\mathrm l}$.  Here a line is said to join $\xi$ to $\eta$ when, for every nearest point $z$ of $o$ on $\alpha_{\xi\eta}$, the two quasigeodesic rays
\[
 [o,z]\cup\alpha_{\xi\eta}[z,\xi),
 \qquad
 [o,z]\cup\alpha_{\xi\eta}[z,\eta)
\]
represent $\xi$ and $\eta$, respectively, in the sublinearly Morse boundary.
\item Constants $\epsilon>0$, $C_{\mathrm v}\ge1$, and $Q\ge1$, and a symmetric function
\[
 d_\Lambda:\Lambda\times\Lambda\longrightarrow[0,\infty)
\]
which vanishes precisely on the diagonal, satisfies the quasi-ultrametric inequality
\[
 d_\Lambda(\xi,\zeta)
 \le Q\max\{d_\Lambda(\xi,\eta),d_\Lambda(\eta,\zeta)\},
\]
and obeys the visual comparison
\[
 C_{\mathrm v}^{-1}
 e^{-\epsilon\rhok(D_o(\xi,\eta))}
 \le d_\Lambda(\xi,\eta)
 \le C_{\mathrm v}
 e^{-\epsilon\rhok(D_o(\xi,\eta))},
 \qquad
 D_o(\xi,\eta):=\dist(o,\alpha_{\xi\eta}).
\]
\end{enumerate}
\end{definition}

We call $d_\Lambda$ uniformly perfect if there are $S>1$ and $r_0>0$ such that, for every $\xi\in\Lambda$ and every $0<r\le r_0$, one can find $\eta\in\Lambda$ with
\[
 \frac rS<d_\Lambda(\xi,\eta)\le r.
\]

\begin{definition}[Radial accessibility through $\Lambda$]\label{def:general-radial-access}
The space $X$ is \emph{$\kappa$--radially accessible through $\Lambda$ from $o$} if there exists $A\ge0$ such that, for every $x\in X$, there are $\xi\in\Lambda$ and $t\ge0$ satisfying
\[
 \dist(x,\sigma_\xi(t))\le A\kappa(|x|).
\]
\end{definition}

\begin{definition}[Center exhaustivity through $\Lambda$]\label{def:stratum-CE}
The space $X$ is \emph{$\kappa$--center--exhaustive through $\Lambda$} if there exists $K\ge0$ such that, for every $x\in X$, there are distinct $\xi,\eta,\zeta\in\Lambda$ with
\[
 \max\bigl\{
 \dist(x,\alpha_{\xi\eta}),
 \dist(x,\alpha_{\xi\zeta}),
 \dist(x,\alpha_{\eta\zeta})
 \bigr\}
 \le K\kappa(|x|).
\]
\end{definition}

\begin{lemma}[Transfer of $\kappa$--scale]\label{lem:kappa-scale-transfer}
For every $A\ge0$ there is $B_A\ge1$ with the following properties.
\begin{enumerate}[label=(\roman*)]
\item If $s,t\ge0$ and $|s-t|\le A\kappa(t)$, then
\[
 \kappa(s)\le B_A\kappa(t).
\]
\item If $x,y\in X$ and $\dist(x,y)\le A\kappa(|x|)$, then
\[
 \kappa(|y|)\le B_A\kappa(|x|).
\]
\item If
\[
 \dist(x,y)\le A\kappa(|x|)
 \quad\text{and}\quad
 \dist(y,z)\le C\kappa(|y|),
\]
then $\dist(x,z)\le(A+CB_A)\kappa(|x|)$.
\end{enumerate}
\end{lemma}

\begin{proof}
By sublinearity, choose $T$ so that $\kappa(t)\le t$ for $t\ge T$.  If $t\ge T$ and $|s-t|\le A\kappa(t)$, then
\[
 s\le t+A\kappa(t)\le(1+A)t,
\]
so Lemma~\ref{lem:kappa-concavity} gives
\[
 \kappa(s)\le(1+A)\kappa(t).
\]
When $t\le T$, the same ratio is bounded because $s\le T+A\max_{[0,T]}\kappa$ and $\kappa\ge1$.  This proves (i).  Part (ii) follows from (i) and
\[
 \bigl||y|-|x|\bigr|\le\dist(x,y).
\]
Part (iii) is the triangle inequality followed by (ii).
\end{proof}

\begin{lemma}[Nearest-point rerooting]\label{lem:nearest-reroot}
There is a constant $M\ge0$, depending only on the common ray gauge $m_{\mathrm r}$, with the following property.  Let $\xi\ne\eta$ lie in $\Lambda$, let $z\in\alpha_{\xi\eta}$ realize
\[
 |z|=D_o(\xi,\eta),
\]
and let $\alpha_{\xi\eta}[z,\xi)$ and $\alpha_{\xi\eta}[z,\eta)$ denote the two tails.  Then each concatenation
\[
 [o,z]\cup\alpha_{\xi\eta}[z,\xi),
 \qquad
 [o,z]\cup\alpha_{\xi\eta}[z,\eta)
\]
is a $(3,0)$--quasigeodesic ray based at $o$.  Moreover,
\[
 \dist(z,\sigma_\xi)\le M\kappa(|z|),
 \qquad
 \dist(z,\sigma_\eta)\le M\kappa(|z|).
\]
\end{lemma}

\begin{proof}
The nearest point $z$ exists because $X$ is proper and the image of a bi-infinite geodesic is closed.  Consider the first concatenation.  It is enough to check two points $u\in[o,z]$ and $v\in\alpha_{\xi\eta}[z,\xi)$.  Since $z$ minimizes distance from $o$ to the line,
\[
 |z|\le |v|\le |u|+\dist(u,v).
\]
As $|z|=|u|+\dist(u,z)$, this gives $\dist(u,z)\le\dist(u,v)$.  Therefore the length in the concatenation between $u$ and $v$ is at most
\[
 \dist(u,z)+\dist(z,v)
 \le \dist(u,z)+\dist(u,v)+\dist(u,z)
 \le3\dist(u,v).
\]
The reverse inequality follows from the metric inequality, so the concatenation is a $(3,0)$--quasigeodesic.  The other tail is identical.

By the meaning of a line joining two boundary points, these quasigeodesic rays represent $\xi$ and $\eta$, respectively.  Apply the common-endpoint stability lemma \cite[Lem.~3.4(i)]{QRT24} to each one and the corresponding geodesic ray $\sigma_\xi$ or $\sigma_\eta$.  The tracking constant depends only on $m_{\mathrm r}$ and the fixed quasigeodesic constants $(3,0)$.  Evaluating at $z$ gives the two displayed estimates.
\end{proof}

\begin{lemma}[Entry at a balanced depth]\label{lem:balanced-entry}
For every $E\ge0$ there are $L_E\ge0$ and $t_E\ge0$ such that the following holds.  If $\xi\ne\eta$ lie in $\Lambda$, $t\ge t_E$, and
\[
 \bigl|\rhok(D_o(\xi,\eta))-\rhok(t)\bigr|\le E,
\]
then
\[
 \dist(\sigma_\xi(t),\alpha_{\xi\eta})
 \le L_E\kappa(t),
 \qquad
 \dist(\sigma_\eta(t),\alpha_{\xi\eta})
 \le L_E\kappa(t),
\]
and
\[
 \dist(\sigma_\xi(t),\sigma_\eta(t))
 \le2L_E\kappa(t).
\]
\end{lemma}

\begin{proof}
Write $D=D_o(\xi,\eta)$ and choose a nearest point $z\in\alpha_{\xi\eta}$.  By Lemma~\ref{lem:rho-two-sided}(ii), after increasing $t_E$ there is $A_E$ such that
\[
 |D-t|\le A_E\kappa(t).
\]
Lemma~\ref{lem:kappa-scale-transfer} then gives
\[
 \kappa(D)\le B_{A_E}\kappa(t).
\]
By Lemma~\ref{lem:nearest-reroot}, choose $u_\xi\in\sigma_\xi$ with
\[
 \dist(z,u_\xi)\le M\kappa(D).
\]
If $s_\xi=|u_\xi|$, then $|s_\xi-D|\le\dist(z,u_\xi)$, and hence
\begin{align*}
 \dist(\sigma_\xi(t),z)
 &\le |t-s_\xi|+\dist(u_\xi,z)\\
 &\le |t-D|+2\dist(u_\xi,z)\\
 &\le\bigl(A_E+2MB_{A_E}\bigr)\kappa(t).
\end{align*}
The same estimate holds for $\eta$.  Taking
\[
 L_E=A_E+2MB_{A_E}
\]
proves the first two inequalities, and the last follows through the common point $z$.
\end{proof}

\begin{lemma}[Balanced boundary triples]\label{lem:UP-balanced-triple}
Let $d_\Lambda$ be a quasi-ultrametric on $\Lambda$ with constant $Q\ge1$.  Fix $\xi\in\Lambda$, $S>1$, and $r>0$.  Suppose that, for every $0<s\le r$, there is $\theta\in\Lambda$ such that
\[
 \frac{s}{S}<d_\Lambda(\xi,\theta)\le s.
\]
Then there are distinct $\eta,\zeta\in\Lambda\setminus\{\xi\}$ such that every pair among $\xi,\eta,\zeta$ satisfies
\[
 c_0r<d_\Lambda(\cdot,\cdot)\le C_0r,
 \qquad
 c_0:=\frac1{2QS^2},
 \quad
 C_0:=Q.
\]
\end{lemma}

\begin{proof}
Apply the annular hypothesis at radius $r$ to choose $\eta$ so that
\[
 \frac rS< a:=d_\Lambda(\xi,\eta)\le r.
\]
Since $0<a/(2Q)\le r$, the same hypothesis applies at the auxiliary radius $a/(2Q)$.  Choose $\zeta$ satisfying
\[
 \frac{a}{2QS}<b:=d_\Lambda(\xi,\zeta)
 \le\frac{a}{2Q}.
\]
In particular, $\zeta$ is distinct from both $\xi$ and $\eta$.  Put $c=d_\Lambda(\eta,\zeta)$.  The quasi-ultrametric inequality gives
\[
 a\le Q\max\{b,c\}.
\]
Since $Qb\le a/2$, it follows that $c\ge a/Q$.  Conversely,
\[
 c\le Q\max\{a,b\}\le Qa.
\]
Combining these estimates with $a>r/S$ yields the claimed common lower and upper bounds.  Thus only the two radii $r$ and $a/(2Q)$, both at the center $\xi$ and both at most $r$, are used.
\end{proof}

\begin{lemma}[Balanced-triangle lemma]\label{lem:balanced-triangle}
Fix numbers $0<c\le C<\infty$.  There are $L\ge0$ and $t_0\ge0$, depending only on the uniform visual data and on $c,C$, such that the following holds.  Let $t\ge t_0$, put
\[
 r_t=e^{-\epsilon\rhok(t)},
\]
and suppose that distinct $\xi,\eta,\zeta\in\Lambda$ satisfy
\[
 cr_t\le d_\Lambda(\xi,\eta),
 d_\Lambda(\xi,\zeta),
 d_\Lambda(\eta,\zeta)
 \le Cr_t.
\]
Then the point $p=\sigma_\xi(t)$ satisfies
\[
 \max\bigl\{
 \dist(p,\alpha_{\xi\eta}),
 \dist(p,\alpha_{\xi\zeta}),
 \dist(p,\alpha_{\eta\zeta})
 \bigr\}
 \le L\kappa(t).
\]
\end{lemma}

\begin{proof}
The visual comparison in Definition~\ref{def:uniform-visual-stratum} gives one constant $E$ such that, for every pair $a\ne b$ among $\xi,\eta,\zeta$,
\[
 \bigl|\rhok(D_o(a,b))-\rhok(t)\bigr|\le E.
\]
Indeed, one may take
\[
 E=\frac1\epsilon
 \max\left\{
 0,
 \log\frac{C_{\mathrm v}}c,
 \log(C_{\mathrm v}C)
 \right\}.
\]
Apply Lemma~\ref{lem:balanced-entry}.  It gives
\[
 \dist(p,\alpha_{\xi\eta}),
 \dist(p,\alpha_{\xi\zeta})
 \le L_E\kappa(t).
\]
Let $p_\eta=\sigma_\eta(t)$.  Applying the same lemma to the pair $(\xi,\eta)$ gives
\[
 \dist(p,p_\eta)\le2L_E\kappa(t),
\]
while applying it to $(\eta,\zeta)$ gives
\[
 \dist(p_\eta,\alpha_{\eta\zeta})
 \le L_E\kappa(t).
\]
Consequently,
\[
 \dist(p,\alpha_{\eta\zeta})
 \le3L_E\kappa(t).
\]
The result follows with $L=3L_E$.
\end{proof}

\begin{theorem}[Radial center criterion]\label{thm:general-radial}
Let $(X,\dist,o)$ be a proper geodesic metric space, and let $\Lambda\subset\bdry X$ carry uniform $\kappa$--visual data based at $o$.  If
\begin{enumerate}[label=(\roman*)]
\item $(\Lambda,d_\Lambda)$ is uniformly perfect, and
\item $X$ is $\kappa$--radially accessible through $\Lambda$ from $o$,
\end{enumerate}
then $X$ is $\kappa$--center--exhaustive through $\Lambda$ for the chosen lines, whose common gauge is $m_{\mathrm l}$.
\end{theorem}

\begin{proof}
Let $S,r_0$ be uniform-perfectness constants and let $A$ be the radial-accessibility constant.  Choose $t_1$ so large that
\[
 e^{-\epsilon\rhok(t)}\le r_0
 \qquad(t\ge t_1)
\]
and that Lemma~\ref{lem:balanced-triangle} applies with the constants $c_0,C_0$ from Lemma~\ref{lem:UP-balanced-triple}.

Fix $x\in X$.  By radial accessibility, choose $\xi\in\Lambda$ and $p=\sigma_\xi(t)$ with
\[
 \dist(x,p)\le A\kappa(|x|).
\]
Suppose first that $t\ge t_1$.  Since $r_t:=e^{-\epsilon\rhok(t)}\le r_0$, uniform perfectness supplies the local annular hypothesis of Lemma~\ref{lem:UP-balanced-triple} at the center $\xi$ for every radius at most $r_t$.  Apply that lemma with cutoff $r_t$ to obtain $\eta,\zeta$ for which all three visual distances are between $c_0r_t$ and $C_0r_t$.  Lemma~\ref{lem:balanced-triangle} gives
\[
 \max_{a\ne b\in\{\xi,\eta,\zeta\}}
 \dist(p,\alpha_{ab})
 \le L\kappa(t).
\]
Since $|t-|x||\le\dist(x,p)\le A\kappa(|x|)$, Lemma~\ref{lem:kappa-scale-transfer} yields
\[
 \kappa(t)\le B_A\kappa(|x|).
\]
Therefore
\[
 \max_{a\ne b}
 \dist(x,\alpha_{ab})
 \le (A+LB_A)\kappa(|x|).
\]

It remains to treat $t<t_1$.  Fix once and for all three distinct points $\xi_0,\eta_0,\zeta_0\in\Lambda$.  The number
\[
 H:=t_1+
 \max\{D_o(\xi_0,\eta_0),D_o(\xi_0,\zeta_0),D_o(\eta_0,\zeta_0)\}
\]
is finite, and $\dist(p,\alpha_{ab})\le H$ for each side of this fixed triple.  Since $\kappa\ge1$,
\[
 \dist(x,\alpha_{ab})
 \le A\kappa(|x|)+H
 \le(A+H)\kappa(|x|).
\]
Taking the larger of the two constants proves the theorem.
\end{proof}

\begin{definition}[Uniform all-basepoint visual family]\label{def:all-basepoint-visual}
Let $\Lambda\subset\bdry X$ contain at least three points.  For each distinct $\xi,\eta\in\Lambda$, choose a line $\alpha_{\xi\eta}=\alpha_{\eta\xi}$ satisfying the joining convention in Definition~\ref{def:uniform-visual-stratum}(ii), and assume that these lines share one $\kappa$--Morse gauge $m_{\mathrm l}$.  A \emph{uniform all-basepoint $\kappa$--visual family} on this line datum consists of constants $\epsilon>0$, $C_{\mathrm v}\ge1$, and $Q\ge1$, together with, for every $x\in X$, a symmetric function
\[
 d_{\Lambda,x}:\Lambda\times\Lambda\longrightarrow[0,\infty)
\]
which vanishes precisely on the diagonal, satisfies
\[
 d_{\Lambda,x}(\xi,\zeta)
 \le Q\max\{d_{\Lambda,x}(\xi,\eta),d_{\Lambda,x}(\eta,\zeta)\},
\]
and obeys the visual comparison
\[
 C_{\mathrm v}^{-1}
 e^{-\epsilon\rhok(D_x(\xi,\eta))}
 \le d_{\Lambda,x}(\xi,\eta)
 \le C_{\mathrm v}
 e^{-\epsilon\rhok(D_x(\xi,\eta))},
 \qquad
 D_x(\xi,\eta):=\dist(x,\alpha_{\xi\eta}),
\]
with the same constants for every basepoint and every pair.
\end{definition}

\begin{definition}[$\kappa$--normalized uniform perfectness over basepoints]\label{def:normalized-basepoint-UP}
Fix $A\ge0$ and $\lambda>0$ and, for $x\in X$, set
\[
 r_{\kappa,A,\lambda}(x)
 :=\lambda\exp\!\bigl(-\epsilon\rhok(A\kappa(|x|))\bigr).
\]
A uniform all-basepoint visual family is \emph{$\kappa$--normalized uniformly perfect over basepoints at parameters $(A,\lambda)$} if there exists $S>1$ such that, for every $x\in X$, every $\xi\in\Lambda$, and every
\[
 0<r\le r_{\kappa,A,\lambda}(x),
\]
there is $\eta\in\Lambda$ satisfying
\[
 \frac rS<d_{\Lambda,x}(\xi,\eta)\le r.
\]
The constants $S$ and $\lambda$ are independent of $x$.  Replacing every $d_{\Lambda,x}$ by $c\,d_{\Lambda,x}$ replaces $\lambda$ by $c\lambda$, so the existence of an admissible $\lambda$ is unchanged by a common rescaling.
\end{definition}

\begin{proposition}[Normalized all-basepoint criterion]\label{prop:normalized-allbasepoint}
Let $(X,\dist,o)$ be a proper geodesic metric space, let $\Lambda\subset\bdry X$ carry a uniform all-basepoint $\kappa$--visual family, and suppose that this family is $\kappa$--normalized uniformly perfect over basepoints at some parameters $A\ge0$ and $\lambda>0$.  Then $X$ is $\kappa$--center--exhaustive through $\Lambda$ for the chosen lines, whose common gauge is $m_{\mathrm l}$.
\end{proposition}

\begin{proof}
Let $S>1$ be the normalized uniform-perfectness constant and put
\[
 c_0=\frac1{2QS^2}.
\]
Fix $x\in X$, write $R=|x|$, and set
\[
 T=A\kappa(R),
 \qquad
 r=r_{\kappa,A,\lambda}(x)=\lambda e^{-\epsilon\rhok(T)}.
\]
Choose any $\xi\in\Lambda$.  By normalized uniform perfectness, the local annular hypothesis of Lemma~\ref{lem:UP-balanced-triple} holds for the quasi-ultrametric $d_{\Lambda,x}$ at this center for every radius $0<s\le r$.  Applying the lemma with cutoff $r$ gives distinct $\eta,\zeta\in\Lambda\setminus\{\xi\}$ such that every pair $a\ne b$ among $\xi,\eta,\zeta$ satisfies
\[
 c_0r<d_{\Lambda,x}(a,b).
\]
The visual upper estimate now gives
\[
 c_0\lambda e^{-\epsilon\rhok(T)}
 <d_{\Lambda,x}(a,b)
 \le C_{\mathrm v}e^{-\epsilon\rhok(D_x(a,b))}.
\]
Consequently,
\[
 \rhok(D_x(a,b))<\rhok(T)+E,
 \qquad
 E:=\frac1\epsilon
 \max\left\{0,\log\frac{C_{\mathrm v}}{c_0\lambda}\right\}.
\]

Let $R_E$ be the threshold in Lemma~\ref{lem:rho-two-sided}(ii).  If $T\ge R_E$, then either $D_x(a,b)\le T$, or else that lemma applies to $D_x(a,b)$ and $T$ and yields
\[
 D_x(a,b)-T
 \le(e-1)E\kappa(T).
\]
Moreover, since $u:=\kappa(R)\ge1$, one has
\[
 \kappa(Au)
 \le C_A u,
 \qquad
 C_A:=\max\{1,A\}\kappa(1).
\]
Indeed, if $Au\le1$, monotonicity gives $\kappa(Au)\le\kappa(1)\le C_A u$; if $Au\ge1$, Lemma~\ref{lem:kappa-concavity}, applied at $1$ with factor $Au$, gives
\[
 \kappa(Au)\le Au\kappa(1)\le C_A u.
\]
Therefore
\[
 \kappa(T)\le C_A\kappa(R).
\]
Thus, in this case,
\[
 D_x(a,b)
 \le\bigl(A+(e-1)EC_A\bigr)\kappa(R).
\]
If $T<R_E$, then
\[
 \rhok(D_x(a,b))<\rhok(R_E)+E,
\]
so
\[
 D_x(a,b)<D_0
 :=\rhok^{-1}\!\bigl(\rhok(R_E)+E\bigr)
 \le D_0\kappa(R),
\]
because $\kappa\ge1$.  Taking
\[
 K=\max\bigl\{D_0,\ A+(e-1)EC_A\bigr\}
\]
proves, simultaneously for the three pairs, that
\[
 \max_{a\ne b\in\{\xi,\eta,\zeta\}}
 \dist(x,\alpha_{ab})
 \le K\kappa(|x|).
\]
\end{proof}

%=============================================================================
\section{Center criteria for trees}\label{sec:trees}
%=============================================================================

Let $T$ be a locally finite simplicial tree with unit-length edges and at least three ends, and write $\partial T$ for its end boundary.  The tree is proper, and its geodesics have a common Morse gauge.  Thus, for every sublinear $\kappa$, its $\kappa$--Morse boundary agrees as a set with $\partial T$.  The correspondence between rooted trees and ultrametric end spaces is classical; see \cite{Hughes04}.

Fix a vertex $o\in T$.  For distinct $\xi,\eta\in\partial T$, let $[\xi,\eta]$ be the unique bi-infinite geodesic joining them and put
\[
 j_o(\xi,\eta):=\dist(o,[\xi,\eta]).
\]
Equivalently, $j_o(\xi,\eta)$ is the length of the common initial segment of the rays $[o,\xi)$ and $[o,\eta)$.

\begin{definition}[Renormalized visual metric]\label{def:tree-visual}
For $\epsilon>0$, define
\[
 d^\kappa_{o,\epsilon}(\xi,\eta)
 :=\exp\!\bigl(-\epsilon\rhok(j_o(\xi,\eta))\bigr)
\]
for $\xi\ne\eta$, and set $d^\kappa_{o,\epsilon}(\xi,\xi)=0$.
\end{definition}

\begin{proposition}\label{prop:tree-ultrametric}
The function $d^\kappa_{o,\epsilon}$ is an ultrametric on $\partial T$.
\end{proposition}

\begin{proof}
Common-prefix lengths satisfy
\[
 j_o(\xi,\zeta)\ge
 \min\{j_o(\xi,\eta),j_o(\eta,\zeta)\}.
\]
Since $\rhok$ is increasing, exponentiation gives
\[
 d^\kappa_{o,\epsilon}(\xi,\zeta)
 \le\max\{d^\kappa_{o,\epsilon}(\xi,\eta),d^\kappa_{o,\epsilon}(\eta,\zeta)\}.
\]
\end{proof}

\subsection{A branching criterion}

For $\xi\in\partial T$, define its set of branching depths from $o$ by
\[
 B_o(\xi):=\{j_o(\xi,\eta):\eta\in\partial T\setminus\{\xi\}\}\subset\N\cup\{0\}.
\]
Thus $n\in B_o(\xi)$ if and only if, at depth $n$ along $[o,\xi)$, there is another infinite direction that separates from $\xi$.

The annular definition below is standard; see \cite{Heinonen01,JarviVuorinen96}.

\begin{definition}\label{def:UP}
A metric space $(Z,d)$ is \emph{uniformly perfect} if there are $S>1$ and $r_0>0$ such that for every $z\in Z$ and every $0<r\le r_0$ there exists $w\in Z$ satisfying
\[
 \frac rS<d(z,w)\le r.
\]
\end{definition}

\begin{theorem}[Branch-depth criterion]\label{thm:UP-criterion}
The metric space $(\partial T,d^\kappa_{o,\epsilon})$ is uniformly perfect if and only if there exist $c>0$ and $u_0\ge0$ such that, for every $\xi\in\partial T$ and every $u\ge u_0$,
\[
 \rhok(B_o(\xi))\cap[u,u+c)\ne\varnothing.
\]
Since $\diam(\partial T,d^\kappa_{o,\epsilon})\le1$, replacing $r_0$ by $\min\{r_0,1\}$ does not change uniform perfectness.  After this replacement, the constants correspond by
\[
 c=\epsilon^{-1}\log S,
 \qquad
 u_0=-\epsilon^{-1}\log r_0.
\]
\end{theorem}

\begin{proof}
Write $r=e^{-\epsilon u}$.  For $n\in B_o(\xi)$, choose $\eta$ with $j_o(\xi,\eta)=n$.  Then
\[
 \frac rS<d^\kappa_{o,\epsilon}(\xi,\eta)\le r
\]
is equivalent to
\[
 e^{-\epsilon(u+c)}<e^{-\epsilon\rhok(n)}\le e^{-\epsilon u},
\]
which in turn is equivalent to
\[
 u\le\rhok(n)<u+c.
\]
The restriction $r\le r_0$ is equivalent to $u\ge u_0$.
\end{proof}

\begin{corollary}[Uniformly eventual branching]\label{cor:regular-UP}
Suppose that there is an integer $N_0\ge0$ such that
\[
 \{N_0,N_0+1,\ldots\}\subset B_o(\xi)
 \qquad\text{for every }\xi\in\partial T.
\]
Then $(\partial T,d^\kappa_{o,\epsilon})$ is uniformly perfect.  In particular, this holds with $S=e^\epsilon$ and $r_0=1$ for a rooted $q$--ary tree, $q\ge2$, and, with any chosen base vertex, for a regular infinite simplicial tree of degree at least $3$.
\end{corollary}

\begin{proof}
Fix $u\ge\rhok(N_0)$ and let $n$ be the least integer with $\rhok(n)\ge u$.  Then $n\ge N_0$, and if $n\ge1$, Lemma~\ref{lem:rho-increments}(i) gives
\[
0\le\rhok(n)-u<\rhok(n)-\rhok(n-1)\le1.
\]
Thus Theorem~\ref{thm:UP-criterion} applies with $c=1$ and $u_0=\rhok(N_0)$.  In each of the two examples in the final sentence of the statement one may take $N_0=0$.  The strict inequality in Definition~\ref{def:UP} is preserved by the minimal choice of $n$; at $r=1$ one may choose an end separating from $\xi$ at $o$.
\end{proof}

\subsection{Ideal medians and center exhaustivity}

For three distinct ends $\xi,\eta,\zeta$, the three lines $[\xi,\eta]$, $[\eta,\zeta]$, and $[\zeta,\xi]$ form a tripod.  Their common point is denoted by
\[
 m(\xi,\eta,\zeta)
\]
and is called the \emph{ideal median}.  Let
\[
 \mathcal M(T):=\{m(\xi,\eta,\zeta):\xi,\eta,\zeta\in\partial T
 \text{ distinct}\}.
\]

\begin{lemma}[Distance to a tripod]\label{lem:tripod-distance}
Let $m=m(\xi,\eta,\zeta)$.  For every $x\in T$,
\[
 \max\bigl\{\dist(x,[\xi,\eta]),\dist(x,[\eta,\zeta]),
 \dist(x,[\zeta,\xi])\bigr\}=\dist(x,m).
\]
\end{lemma}

\begin{proof}
Let $P$ be the union of the three sides, and let $p$ be the nearest-point projection of $x$ to the convex subtree $P$.  If $p=m$, each side is at distance $\dist(x,m)$.  Otherwise $p$ lies on one arm of the tripod.  Two sides contain that arm and are at distance $\dist(x,p)$ from $x$, whereas the third side meets the arm only at $m$ and is at distance
\[
 \dist(x,p)+\dist(p,m)=\dist(x,m).
\]
The maximum is therefore $\dist(x,m)$.
\end{proof}

Fix the original basepoint $o$ and write $|x|=\dist(o,x)$.  We say that $x$ is a $K$--center of the ideal triple $(\xi,\eta,\zeta)$ if it lies within $K\kappa(|x|)$ of each of the three sides.  The tree is \emph{$\kappa$--center--exhaustive} if one constant $K$ works for every $x\in T$ and some ideal triple depending on $x$.

\begin{proposition}[Median criterion for center exhaustivity]\label{prop:tree-CE}
A locally finite tree $T$ with at least three ends is $\kappa$--center--exhaustive if and only if
\[
 \sup_{x\in T}\frac{\dist(x,\mathcal M(T))}{\kappa(|x|)}<\infty.
\]
\end{proposition}

\begin{proof}
If $x$ is a $K$--center of $(\xi,\eta,\zeta)$, Lemma~\ref{lem:tripod-distance} gives
\[
\dist\bigl(x,m(\xi,\eta,\zeta)\bigr)\le K\kappa(|x|).
\]
Conversely, $\mathcal M(T)$ is nonempty because $T$ has at least three ends.  It is a subset of the $1$--separated vertex set and is therefore closed.  Since the locally finite tree $T$ is proper, the distance from $x$ to $\mathcal M(T)$ is attained by some median $m\in\mathcal M(T)$.  If $\dist(x,m)\le K\kappa(|x|)$, choose an ideal triple with median $m$.  Lemma~\ref{lem:tripod-distance} then says that $x$ is a $K$--center of that triple.
\end{proof}

\subsection{Uniform perfectness, radial accessibility, and centers}

Let
\[
 \mathcal R_o(T):=\bigcup_{\xi\in\partial T}[o,\xi)
\]
be the \emph{rooted ray core} of $T$.  It is a closed subtree.

\begin{definition}[Radial accessibility]\label{def:radial-accessibility}
The tree $T$ is \emph{$\kappa$--radially accessible from $o$} if there is $A\ge0$ such that for every $x\in T$ there are $\xi\in\partial T$ and $p\in[o,\xi)$ satisfying
\[
 \dist(x,p)\le A\kappa(|x|).
\]
Equivalently,
\[
 \sup_{x\in T}\frac{\dist(x,\mathcal R_o(T))}{\kappa(|x|)}<\infty.
\]
\end{definition}

\begin{theorem}[Tree criterion]\label{thm:radial-equivalence}
Let $T$ be a locally finite simplicial tree with unit-length edges, at least three ends, and base vertex $o$.  The following are equivalent.
\begin{enumerate}[label=(\roman*)]
\item $T$ is $\kappa$--center--exhaustive.
\item $T$ is $\kappa$--radially accessible from $o$, and $(\partial T,d^\kappa_{o,\epsilon})$ is uniformly perfect for some $\epsilon>0$.
\item $T$ is $\kappa$--radially accessible from $o$, and $(\partial T,d^\kappa_{o,\epsilon})$ is uniformly perfect for every $\epsilon>0$.
\end{enumerate}
\end{theorem}

\begin{proof}
Assume first that $T$ is $\kappa$--center--exhaustive with constant $K$.  Every ideal median lies on a ray from $o$: if $m=m(\xi,\eta,\zeta)$, at most one of the three infinite components of $T\setminus\{m\}$ determined by the ideal tripod contains $o$, so an end in either of the other components determines a ray from $o$ through $m$.  Hence
\[
 \mathcal M(T)\subset\mathcal R_o(T).
\]
Proposition~\ref{prop:tree-CE} therefore gives $\kappa$--radial accessibility with constant $K$.

We next prove fixed-basepoint uniform perfectness.  Fix $\xi\in\partial T$ and let $v_n$ be the vertex of $[o,\xi)$ at depth $n$.  Choose an ideal median $m$ with
\[
 \dist(v_n,m)\le K\kappa(n),
\]
and let $q$ be the projection of $m$ to $[o,\xi)$.  The point $q$ is a vertex.  If $q\ne m$, the component of $T\setminus\{q\}$ containing $m$ contains an end: among the three end-components at the ideal median $m$, at most one points from $m$ back toward $q$, so at least two remain in that component.  If $q=m$, choose an infinite component at $m$ different from the one containing $\xi$ and, when $o\ne m$, from the one containing $o$.  In either case there is $\eta\ne\xi$ such that
\[
 j_o(\xi,\eta)=\dist(o,q)=:b.
\]
Thus $b\in B_o(\xi)$, and
\[
 |b-n|\le\dist(v_n,m)\le K\kappa(n).
\]
By Lemma~\ref{lem:rho-two-sided}(i), for all sufficiently large $n$,
\[
 |\rhok(b)-\rhok(n)|\le2K.
\]
Set $C=2K$.  Given sufficiently large $u$, choose the least integer $n$ with
\[
 \rhok(n)\ge u+C.
\]
Then Lemma~\ref{lem:rho-increments}(i) gives
\[
 u+C\le\rhok(n)<u+C+1.
\]
The branch depth $b$ found above satisfies
\[
 u\le\rhok(b)<u+2C+1.
\]
Theorem~\ref{thm:UP-criterion} now proves uniform perfectness for every $\epsilon>0$.  This establishes (i)$\Rightarrow$(iii), and (iii)$\Rightarrow$(ii) is immediate.

Assume conversely that (ii) holds.  By Theorem~\ref{thm:UP-criterion}, there are $c>0$ and $u_0\ge0$ such that
\[
 \rhok(B_o(\xi))\cap[u,u+c)\ne\varnothing
\]
for every $\xi\in\partial T$ and every $u\ge u_0$.  We first show that points of the rooted ray core lie $O(\kappa)$ from ideal medians.  Let $p\in[o,\xi)$ and put $t=|p|$.  Choose $t_*>0$ so that, whenever $t\ge t_*$,
\[
 \rhok(t)-c\ge u_0
\]
and Lemma~\ref{lem:rho-two-sided}(ii) applies with $C=c$.  First suppose that $t\ge t_*$.  Apply the branch-depth criterion at $u=\rhok(t)-c$ and at $u=\rhok(t)$ to obtain
\[
 b_-,b_+\in B_o(\xi)
\]
with
\[
 \rhok(t)-c\le\rhok(b_-)<\rhok(t)
 \le\rhok(b_+)<\rhok(t)+c.
\]
Thus $b_-<t\le b_+$.  Choose ends $\zeta,\eta$ satisfying
\[
 j_o(\xi,\zeta)=b_-,
 \qquad
 j_o(\xi,\eta)=b_+.
\]
At the branch point $q_+\in[o,\xi)$ of depth $b_+$, the components toward $\xi$ and $\eta$ are infinite, and the component toward $o$ contains the end $\zeta$ through the earlier branch at depth $b_-$.  Hence
\[
 q_+=m(\xi,\eta,\zeta)\in\mathcal M(T).
\]
Lemma~\ref{lem:rho-two-sided}(ii) gives
\[
 \dist(p,q_+)=b_+-t\le(e-1)c\kappa(t).
\]
For $0\le t<t_*$, fix one $m_0\in\mathcal M(T)$.  Then
\[
 \dist(p,\mathcal M(T))\le\dist(p,m_0)
 \le t_*+|m_0|
 \le(t_*+|m_0|)\kappa(t),
\]
because $\kappa\ge1$.  Consequently there is $B\ge0$ such that
\[
 \dist(p,\mathcal M(T))\le B\kappa(|p|)
 \qquad\text{for every }p\in\mathcal R_o(T).
\]

Let $A$ be a radial-accessibility constant and fix $x\in T$, with $R=|x|$.  Choose $p\in\mathcal R_o(T)$ such that
\[
 \dist(x,p)\le A\kappa(R),
\]
and put $t=|p|$.  Choose $R_*>0$ so that $A\kappa(R)\le R$ whenever $R\ge R_*$.  For such $R$, one has $t\le2R$ and therefore
\[
 \kappa(t)\le2\kappa(R)
\]
by Lemma~\ref{lem:kappa-concavity}.  It follows that
\[
 \dist(x,\mathcal M(T))
 \le \dist(x,p)+\dist(p,\mathcal M(T))
 \le (A+2B)\kappa(R).
\]
For $R<R_*$, the fixed median $m_0$ above gives
\[
 \dist(x,\mathcal M(T))\le\dist(x,m_0)
 \le R_*+|m_0|
 \le(R_*+|m_0|)\kappa(R).
\]
Proposition~\ref{prop:tree-CE} proves that $T$ is $\kappa$--center--exhaustive.  Thus (ii)$\Rightarrow$(i).
\end{proof}

\subsection{Uniformity over basepoints}

Let
\[
 \mathcal C(T):=\bigcup_{\xi\ne\eta\in\partial T}[\xi,\eta]
\]
be the geodesic core of $T$.  For an arbitrary point $x\in T$ and distinct $\xi,\eta\in\partial T$, define
\[
 d^\kappa_{x,\epsilon}(\xi,\eta)
 =\exp\!\bigl(-\epsilon\rhok(\dist(x,[\xi,\eta]))\bigr),
\]
and set $d^\kappa_{x,\epsilon}(\xi,\xi)=0$.
This is an ultrametric even when $x$ lies outside the geodesic core, because
\[
 j_x(\xi,\eta):=\dist(x,[\xi,\eta])
\]
is exactly the length of the common initial segment of the rays $[x,\xi)$ and $[x,\eta)$.  Hence
\[
 j_x(\xi,\zeta)\ge
 \min\{j_x(\xi,\eta),j_x(\eta,\zeta)\},
\]
and monotonicity of $\rhok$ gives the ultrametric inequality.

\begin{definition}[Uniform perfectness over basepoints]\label{def:basepoint-UP}
The family $\{(\partial T,d^\kappa_{x,\epsilon}):x\in T\}$ is \emph{uniformly perfect over basepoints} if there exist $S>1$ and $r_0>0$, independent of $x$, such that Definition~\ref{def:UP} holds for every metric $d^\kappa_{x,\epsilon}$ with the same $S,r_0$.
\end{definition}

\begin{proposition}[Uniform cutoff and the geodesic core]\label{prop:basepoint-core}
If the family in Definition~\ref{def:basepoint-UP} is uniformly perfect over basepoints, then
\[
 \sup_{x\in T}\dist(x,\mathcal C(T))<\infty.
\]
More precisely, with $\bar r_0:=\min\{r_0,1\}$,
\[
 \dist(x,\mathcal C(T))
 <\rhok^{-1}\!\left(\frac1\epsilon\log\frac{S}{\bar r_0}\right)
 \qquad\text{for every }x\in T.
\]
\end{proposition}

\begin{proof}
Replacing $r_0$ by $\bar r_0$ preserves the hypothesis.  Fix $x\in T$ and $\xi\in\partial T$.  Apply Definition~\ref{def:basepoint-UP} with $r=\bar r_0$ to obtain $\eta\ne\xi$ such that
\[
 d^\kappa_{x,\epsilon}(\xi,\eta)>\frac{\bar r_0}{S}.
\]
Since $[\xi,\eta]\subset\mathcal C(T)$,
\[
 \frac{\bar r_0}{S}
 <e^{-\epsilon\rhok(\dist(x,[\xi,\eta]))}
 \le e^{-\epsilon\rhok(\dist(x,\mathcal C(T)))}.
\]
Taking logarithms and applying the increasing inverse of $\rhok$ gives the claim.
\end{proof}

\begin{corollary}[Normalized all-basepoint criterion for trees]\label{cor:tree-normalized}
Let $T$ be a locally finite tree with at least three ends and basepoint $o$.  Suppose that, for some $A\ge0$ and $S>1$, every $x\in T$, every $\xi\in\partial T$, and every
\[
 0<r\le
 \exp\!\bigl(-\epsilon\rhok(A\kappa(|x|))\bigr)
\]
admit $\eta\ne\xi$ with
\[
 \frac rS<d^\kappa_{x,\epsilon}(\xi,\eta)\le r.
\]
Then $T$ is $\kappa$--center--exhaustive.
\end{corollary}

\begin{proof}
Use Proposition~\ref{prop:normalized-allbasepoint} with parameters $(A,\lambda)=(A,1)$, the unique connecting lines, and $C_{\mathrm v}=Q=1$.
\end{proof}

%=============================================================================
\section{Dead-end examples}\label{sec:counterexample}
%=============================================================================

\begin{theorem}\label{thm:counterexample}
Let $\kappa\colon[0,\infty)\to[1,\infty)$ be increasing, concave, and sublinear.  There exists a locally finite tree $X$ with basepoint $o$ such that:
\begin{enumerate}[label=(\alph*)]
\item every point of $X$ is the initial point of a geodesic ray and all geodesics share one Morse gauge;
\item $(\partial X,d^\kappa_{o,\epsilon})$ is uniformly perfect for every $\epsilon>0$;
\item $X$ is not $\kappa$--radially accessible from $o$;
\item $X$ is not $\kappa$--center--exhaustive;
\item the family $\{(\partial X,d^\kappa_{x,\epsilon}):x\in X\}$ is not uniformly perfect over basepoints.
\end{enumerate}
\end{theorem}

\begin{proof}
Let $T$ be a $3$--regular infinite simplicial tree with basepoint $o$.  Fix a geodesic ray
\[
 o=v_0,v_1,v_2,\ldots
\]
in $T$.  At each $v_n$, $n\ge1$, attach a finite simplicial segment $I_n$ of length $n$, meeting $T$ only at $v_n$.  Let $x_n$ be the other endpoint of $I_n$, and let $X$ be the resulting tree.

The tree $X$ is locally finite and has the same end boundary as $T$.  All its geodesics share a Morse gauge, and for each $x\in X$ and $\xi\in\partial X$ there is a unique ray $[x,\xi)$.  This proves (a).

Every bi-infinite geodesic between two ends lies entirely in the original core $T$.  Consequently, for $\xi,\eta\in\partial X=\partial T$,
\[
 \dist_X(o,[\xi,\eta])=\dist_T(o,[\xi,\eta]),
\]
so the fixed-basepoint metric on $\partial X$ is exactly the renormalized visual metric on the $3$--regular core $T$.  Corollary~\ref{cor:regular-UP} proves (b), with uniform perfectness constant $S=e^\epsilon$.

The rooted ray core $\mathcal R_o(X)$ is exactly the original $3$--regular tree $T$: a ray from $o$ cannot enter one of the finite attached segments.  Since $|x_n|=\dist(o,x_n)=2n$,
\[
 \frac{\dist(x_n,\mathcal R_o(X))}{\kappa(|x_n|)}
 =\frac{n}{\kappa(2n)}\longrightarrow\infty.
\]
This proves (c).

Every ideal triangle is contained in $T$, and therefore every one of its sides is at distance at least $n$ from $x_n$.  If $X$ were $\kappa$--center--exhaustive with constant $K$, then
\[
 n\le K\kappa(2n)
\]
for all $n$, again contradicting sublinearity.  This proves (d).

Finally, $\dist(x_n,\mathcal C(X))=n\to\infty$.  Proposition~\ref{prop:basepoint-core} proves (e).
\end{proof}

Assume now that $\kappa$ is unbounded.

\begin{theorem}[Sublinear dead ends]\label{thm:allbasepoint-obstruction}
Assume that $\kappa$ is unbounded.  There is a locally finite tree $Y$ with basepoint $o$ such that:
\begin{enumerate}[label=(\alph*)]
\item all geodesics in $Y$ have a common Morse gauge, and $Y$ is $\kappa$--radially accessible from $o$;
\item $(\partial Y,d^\kappa_{o,\epsilon})$ is uniformly perfect for every $\epsilon>0$, and hence $Y$ is $\kappa$--center--exhaustive;
\item for every $\epsilon>0$ there are basepoints $y_n\in Y$ for which
\[
 \diam(\partial Y,d^\kappa_{y_n,\epsilon})\longrightarrow0.
\]
Consequently, the family of boundary metrics is not uniformly perfect with a basepoint-independent absolute upper radius.
\end{enumerate}
\end{theorem}

\begin{proof}
Start with the same $3$--regular tree $T$ and ray $o=v_0,v_1,\ldots$.  At $v_n$ attach a finite segment $J_n$ of integer length
\[
 h_n=\lfloor\kappa(n)\rfloor,
\]
and let $y_n$ be its endpoint outside $T$.  Denote the resulting tree by $Y$.  It is locally finite, and its geodesics have a common Morse gauge.

The rooted ray core $\mathcal R_o(Y)$ is the original tree $T$.  If $z\in J_n$ lies at distance $s$ from $v_n$, then $|z|=n+s$ and
\[
 \dist(z,\mathcal R_o(Y))=s\le h_n\le\kappa(n)\le\kappa(n+s)=\kappa(|z|).
\]
Thus $Y$ is $\kappa$--radially accessible with constant $1$, proving (a).  Every line joining two ends is contained in $T$, so the fixed-basepoint boundary metric is exactly that of the $3$--regular core.  Corollary~\ref{cor:regular-UP} gives uniform perfectness for every $\epsilon>0$, and Theorem~\ref{thm:radial-equivalence} gives $\kappa$--center--exhaustivity.

For the tip $y_n$, every connecting line is contained in $T$, hence lies at distance at least $h_n$; some such lines pass through $v_n$.  Hence
\[
 \diam(\partial Y,d^\kappa_{y_n,\epsilon})
 =e^{-\epsilon\rhok(h_n)}.
\]
Because $\kappa$ is unbounded and increasing, $h_n\to\infty$; because $\rhok$ is unbounded, the displayed diameters tend to zero.  If constants $S>1$ and $r_0>0$ worked for all basepoints, then for large $n$ the diameter would be smaller than $r_0/S$, contradicting Definition~\ref{def:UP} at radius $r_0$.
\end{proof}

\begin{proposition}[Normalized cutoff in the examples]\label{prop:normalized-examples}
Fix $\epsilon>0$.
\begin{enumerate}[label=(\roman*)]
\item The tree $Y$ of Theorem~\ref{thm:allbasepoint-obstruction} is $\kappa$--normalized uniformly perfect over basepoints at parameters $(A,\lambda)=(1,1)$, with every $S>e^\epsilon$.
\item The tree $X$ of Theorem~\ref{thm:counterexample} does not satisfy the normalized condition for any fixed $A\ge0$, $\lambda>0$, and $S>1$.
\end{enumerate}
\end{proposition}

\begin{proof}
Let $T$ be the $3$--regular core of $Y$ and set $h(x)=\dist(x,T)$.  The proof of Theorem~\ref{thm:allbasepoint-obstruction} gives
\[
 h(x)\le\kappa(|x|).
\]
Fix $x\in Y$, $\xi\in\partial Y$, and
\[
 0<r\le e^{-\epsilon\rhok(\kappa(|x|))}.
\]
Write $r=e^{-\epsilon\rhok(u)}$.  Then $u\ge\kappa(|x|)\ge h(x)$.  After the ray $[x,\xi)$ enters $T$, its branching depths have gaps at most $1$, so some branching depth $b$ lies in $[u,u+1]$.  Choose $\eta$ separating from $\xi$ at depth $b$.  Since $\rhok(b)-\rhok(u)\le b-u\le1$,
\[
 e^{-\epsilon}r
 \le d^\kappa_{x,\epsilon}(\xi,\eta)
 \le r.
\]
This proves (i).

For the tips $x_n$ of the length-$n$ dead ends in $X$, one has $|x_n|=2n$ and
\[
 \diam(\partial X,d^\kappa_{x_n,\epsilon})
 =e^{-\epsilon\rhok(n)}.
\]
Fix $A\ge0$ and $\lambda>0$.  For all sufficiently large $n$, sublinearity gives $A\kappa(2n)<n$, and
\[
 \rhok(n)-\rhok(A\kappa(2n))
 \ge\frac{n-A\kappa(2n)}{\kappa(n)}
 \ge\frac{n}{\kappa(n)}-2A
 \longrightarrow\infty,
\]
where Lemma~\ref{lem:kappa-concavity} gives $\kappa(2n)\le2\kappa(n)$.  Hence
\[
 \frac{\diam(\partial X,d^\kappa_{x_n,\epsilon})}
 {r_{\kappa,A,\lambda}(x_n)}
 \longrightarrow0.
\]
For any fixed $S>1$, the annulus required at radius $r_{\kappa,A,\lambda}(x_n)$ is empty for all sufficiently large $n$.  This proves (ii).
\end{proof}

%=============================================================================
\section{Quasisymmetric changes of scale}\label{sec:scale}
%=============================================================================

Metric-preserving functions are discussed, for example, in \cite{Corazza99}.
Here we use the following restricted class.

\begin{definition}[Metric transform]
An increasing homeomorphism $\phi\colon[0,\infty)\to[0,\infty)$ with $\phi(0)=0$ is called a \emph{metric transform} here if it is subadditive:
\[
 \phi(s+t)\le\phi(s)+\phi(t).
\]
Then $\phi\circ d$ is a metric whenever $d$ is a metric.
\end{definition}

Every increasing concave homeomorphism with $\phi(0)=0$ is subadditive, by the argument of Lemma~\ref{lem:kappa-concavity}.

Recall the metric definition of Tukia and V\"ais\"al\"a \cite{TukiaVaisala80}; see also \cite{Heinonen01}.  A homeomorphism $f:(Z,d)\to(W,d')$ is $\eta$--quasisymmetric if $\eta:[0,\infty)\to[0,\infty)$ is an increasing homeomorphism and
\[
 d(x,a)\le t\,d(x,b)
 \quad\Longrightarrow\quad
 d'(f(x),f(a))\le\eta(t)\,d'(f(x),f(b))
\]
for all distinct $x,a,b\in Z$ and all $t>0$.

For a metric transform $\phi$, define its relative-scale function
\[
 \omega_\phi(t):=\sup_{r>0}\frac{\phi(tr)}{\phi(r)},
 \qquad t>0.
\]

\begin{lemma}[Automatic control for metric transforms]\label{lem:omega-automatic}
For every metric transform $\phi$,
\[
 \omega_\phi(t)\le
 \begin{cases}
  1,&0<t\le1,\\
  \lceil t\rceil,&t>1.
 \end{cases}
\]
In particular, $\omega_\phi(t)<\infty$ for every $t>0$, and
\[
 \phi(2r)\le2\phi(r)
 \qquad(r\ge0).
\]
\end{lemma}

\begin{proof}
If $0<t\le1$, monotonicity gives $\phi(tr)\le\phi(r)$.  If $t>1$ and $n=\lceil t\rceil$, then monotonicity and repeated subadditivity give
\[
 \phi(tr)\le\phi(nr)\le n\phi(r).
\]
Taking the supremum over $r>0$ proves the claim.
\end{proof}

\begin{proposition}[Relative-scale criterion]\label{prop:relative-scale}
Let $\phi$ be a metric transform.  The following are equivalent.
\begin{enumerate}[label=(\roman*)]
\item
\[
 \lim_{t\downarrow0}\omega_\phi(t)=0.
\]
\item There is an increasing homeomorphism $\eta:[0,\infty)\to[0,\infty)$ such that, for every metric space $(Z,d)$, the identity
\[
 \id:(Z,d)\longrightarrow(Z,\phi\circ d)
\]
is $\eta$--quasisymmetric.
\item There is an increasing homeomorphism $\eta:[0,\infty)\to[0,\infty)$ such that the identity is $\eta$--quasisymmetric for every three-point metric space.
\end{enumerate}
Moreover, in (ii) the distortion function may be chosen using only $\omega_\phi$.
\end{proposition}

\begin{proof}
Assume (i).  Let $x,a,b\in Z$ be distinct and suppose $d(x,a)\le t\,d(x,b)$.  Monotonicity gives
\[
 \frac{\phi(d(x,a))}{\phi(d(x,b))}
 \le\frac{\phi(t\,d(x,b))}{\phi(d(x,b))}
 \le\omega_\phi(t).
\]
The function $\omega_\phi$ is finite by Lemma~\ref{lem:omega-automatic} and is nondecreasing.  Hence it is Borel measurable, locally bounded, and locally integrable.  A concrete increasing homeomorphism majorant is obtained by setting
\[
 \eta(0)=0,
 \qquad
 \eta(t)=t+\int_1^2\omega_\phi(tu)\,du
 \quad(t>0).
\]
The integral term is finite and nondecreasing.  If $W(t)=\int_0^t\omega_\phi(v)\,dv$, then it equals
\[
 \frac{W(2t)-W(t)}{t},
\]
which is continuous for $t>0$ because $W$ is continuous.  Moreover,
\[
 \omega_\phi(t)\le\int_1^2\omega_\phi(tu)\,du\le\omega_\phi(2t)
\]
shows that it majorizes $\omega_\phi(t)$ and tends to $0$ as $t\downarrow0$.  The added term $t$ makes $\eta$ strictly increasing and unbounded.  Thus $\eta$ is an increasing homeomorphism of $[0,\infty)$, and the preceding ratio estimate proves (ii).  The implication (ii)$\Rightarrow$(iii) is immediate.

Assume (iii).  Fix $r>0$ and $t>0$.  On the three-point set $\{x,a,b\}$ define
\[
 d(x,a)=tr,\qquad d(x,b)=r,\qquad d(a,b)=\max\{tr,r\}.
\]
This is a metric.  Quasisymmetry gives
\[
 \frac{\phi(tr)}{\phi(r)}\le\eta(t).
\]
Taking the supremum over $r>0$ yields $\omega_\phi(t)\le\eta(t)$.  Since $\eta(t)\to0$ as $t\downarrow0$, condition (i) follows.
\end{proof}

\begin{corollary}[Power relative control]\label{cor:power-transform}
Let $\phi$ be a metric transform.  Suppose that for some $C\ge1$ and $\alpha>0$,
\[
 \frac{\phi(tr)}{\phi(r)}\le C t^\alpha
 \qquad(0<t\le1,\ r>0).
\]
Then $\id:(Z,d)\to(Z,\phi\circ d)$ is quasisymmetric for every metric space $Z$, with a distortion function independent of $Z$.
\end{corollary}

\begin{proof}
The displayed estimate gives $\omega_\phi(t)\le Ct^\alpha$ for $0<t\le1$, so Proposition~\ref{prop:relative-scale} applies.  No additional assumption for $t\ge1$ is needed by Lemma~\ref{lem:omega-automatic}.
\end{proof}

\begin{proposition}[A concave transform without relative decay]\label{prop:doubling-counterexample}
There exists an increasing concave homeomorphism $\phi\colon[0,\infty)\to[0,\infty)$ with $\phi(0)=0$ such that, for every fixed $\tau\in(0,1)$,
\[
 \lim_{r\downarrow0}\frac{\phi(\tau r)}{\phi(r)}=1.
\]
The function $\phi$ is a metric transform and satisfies $\phi(2t)\le2\phi(t)$, but the identity
\[
 ([0,e^{-2}],|\cdot|)\longrightarrow([0,e^{-2}],\phi\circ|\cdot|)
\]
is not quasisymmetric.
\end{proposition}

\begin{proof}
Set $R=e^{-2}$ and define
\[
 \phi(t)=
 \begin{cases}
 0,&t=0,\\[1mm]
 \displaystyle\frac{1}{\log(e/t)},&0<t\le R,\\[3mm]
 \displaystyle\frac13+\frac{e^2}{9}(t-R),&t\ge R.
 \end{cases}
\]
For $0<t\le R$, writing $L=\log(e/t)$ gives
\[
 \phi'(t)=\frac{1}{tL^2},
 \qquad
 \phi''(t)=\frac{2-L}{t^2L^3}<0,
\]
because $L\ge3$.  Moreover, $\phi(t)\to0$ as $t\downarrow0$.  A concave function on $(0,R]$ with a finite right limit remains concave after assigning that limit at $0$: apply the concavity inequality with one endpoint $x_n\downarrow0$ and pass to the limit.  Thus the displayed definition is concave on $[0,R]$.  At $t=R$, the value and derivative agree with the linear extension; the derivative is decreasing on $(0,R]$ and then constant.  Hence the extension is concave on all of $[0,\infty)$.  It is also increasing, continuous, and unbounded, and is therefore a metric transform.  The inequality $\phi(2t)\le2\phi(t)$ also follows from Lemma~\ref{lem:omega-automatic}.

Fix $\tau\in(0,1)$.  For $r\downarrow0$,
\[
 \frac{\phi(\tau r)}{\phi(r)}
 =\frac{\log(e/r)}{\log(e/r)+\log(1/\tau)}
 \longrightarrow1.
\]
If the identity on $[0,R]$ were $\eta$--quasisymmetric, choose $\tau$ so small that $\eta(\tau)<1/2$.  Applying quasisymmetry to $x=0$, $a=\tau r$, and $b=r$ would give
\[
 \frac{\phi(\tau r)}{\phi(r)}\le\eta(\tau)<\frac12,
\]
contradicting the preceding limit for sufficiently small $r$.
\end{proof}

%=============================================================================
\section{Rooted \texorpdfstring{$q$}{q}--ary trees: quasisymmetry, doubling, and dimension}\label{sec:regular}
%=============================================================================

Let $T_q$ be the rooted $q$--ary tree, where $q\ge2$.  Its boundary can be identified with the sequence space
\[
 \Sigma_q=\{1,\ldots,q\}^{\N}.
\]
For distinct $\xi,\eta\in\Sigma_q$, let $\ell(\xi,\eta)$ be the number of common initial symbols.  Fix $a,\epsilon>0$ and define
\[
 d_a(\xi,\eta)=e^{-a\ell(\xi,\eta)},
 \qquad
 d_\kappa(\xi,\eta)=e^{-\epsilon\rhok(\ell(\xi,\eta))}.
\]
Both are compact ultrametrics.

\begin{theorem}[Renormalized metrics on rooted $q$--ary trees]\label{thm:regular-dichotomy}
Let $K_\infty=\lim_{t\to\infty}\kappa(t)\in[1,\infty]$.
\begin{enumerate}[label=(\roman*)]
\item $(\Sigma_q,d_\kappa)$ is uniformly perfect for every $\kappa$, with $S=e^\epsilon$.
\item The identity $\id:(\Sigma_q,d_a)\to(\Sigma_q,d_\kappa)$ is quasisymmetric if and only if $K_\infty<\infty$.
\item The metric space $(\Sigma_q,d_\kappa)$ is doubling if and only if $K_\infty<\infty$.
\item Its Hausdorff dimension is
\[
 \Hdim(\Sigma_q,d_\kappa)
 =\begin{cases}
 \displaystyle\frac{K_\infty\log q}{\epsilon},&K_\infty<\infty,\\[3mm]
 \infty,&K_\infty=\infty.
 \end{cases}
\]
\end{enumerate}
\end{theorem}

\begin{proof}
Part (i) is Corollary~\ref{cor:regular-UP}.

Assume first that $K_\infty<\infty$ and put $M=K_\infty$.  Then $1\le\kappa\le M$.  Let $x,u,v\in\Sigma_q$ be distinct and write
\[
 p=\ell(x,u),\qquad q_0=\ell(x,v).
\]
Suppose $d_a(x,u)\le t\,d_a(x,v)$.
If $0<t\le1$, then
\[
 p-q_0\ge\frac1a\log\frac1t.
\]
Lemma~\ref{lem:rho-increments}(iii) gives
\[
 \frac{d_\kappa(x,u)}{d_\kappa(x,v)}
 =e^{-\epsilon(\rhok(p)-\rhok(q_0))}
 \le e^{-\epsilon(p-q_0)/M}
 \le t^{\epsilon/(aM)}.
\]
If $t\ge1$ and $p<q_0$, then
\[
 q_0-p\le\frac1a\log t
\]
and, using $\rhok(q_0)-\rhok(p)\le q_0-p$,
\[
 \frac{d_\kappa(x,u)}{d_\kappa(x,v)}
 \le t^{\epsilon/a}.
\]
If $p\ge q_0$, the same ratio is at most $1$.  Thus the identity is quasisymmetric with distortion
\[
 \eta(t)=
 \begin{cases}
 t^{\epsilon/(aM)},&0\le t\le1,\\
 t^{\epsilon/a},&t\ge1.
 \end{cases}
\]

Now assume $K_\infty=\infty$ and suppose, for contradiction, that the identity is $\eta$--quasisymmetric.  Choose an integer $L\ge1$ so large that
\[
 \eta(e^{-aL})<\frac12.
\]
For every $n$ one can choose $x,u_n,v_n\in\Sigma_q$ with
\[
 \ell(x,u_n)=n+L,
 \qquad
 \ell(x,v_n)=n.
\]
Then
\[
 \frac{d_a(x,u_n)}{d_a(x,v_n)}=e^{-aL},
\]
whereas Lemma~\ref{lem:rho-increments}(ii) gives
\[
 \frac{d_\kappa(x,u_n)}{d_\kappa(x,v_n)}
 =e^{-\epsilon(\rhok(n+L)-\rhok(n))}\longrightarrow1.
\]
This contradicts quasisymmetry.  Part (ii) follows.

For (iii), first suppose $\kappa\le M$.  Choose $L\in\N$ so that
\[
 \frac{\epsilon L}{M}\ge\log2.
\]
Put $r_n=e^{-\epsilon\rhok(n)}$.  A depth-$n$ cylinder is exactly a closed ball of radius $r_n$ and has diameter $r_n$.  If $0<R<1$ and $r_n\le R<r_{n-1}$, then every closed ball of radius $R$ is a depth-$n$ cylinder.  Such a cylinder is the disjoint union of $q^L$ depth-$(n+L)$ cylinders, each of diameter at most $r_n/2\le R/2$ by Lemma~\ref{lem:rho-increments}(iii).  If $R\ge1$, the radius-$R$ ball is the whole depth-$0$ cylinder, and its depth-$L$ subcylinders likewise have diameter at most $1/2\le R/2$.  Hence every radius-$R$ ball is covered by at most $q^L$ radius-$R/2$ balls, so $q^L$ is a uniform doubling bound.

Conversely, suppose $\kappa$ is unbounded.  Given $L\ge1$, choose $n$ so large that
\[
 \frac{\epsilon L}{\kappa(n)}<\log2.
\]
Then Lemma~\ref{lem:rho-increments} gives
\[
 r_{n+L}=e^{-\epsilon\rhok(n+L)}>
 \frac12e^{-\epsilon\rhok(n)}=\frac{r_n}{2}.
\]
A ball of radius $r_n/2$ that meets a depth-$n$ cylinder is contained in that cylinder, and it cannot meet two distinct depth-$(n+L)$ subcylinders: two such points are at distance at least $r_{n+L}>r_n/2$.  Thus covering one depth-$n$ cylinder by radius-$r_n/2$ balls requires at least $q^L$ balls.  Since $L$ is arbitrary, no doubling constant exists.

For (iv), retain the notation $r_n=e^{-\epsilon\rhok(n)}$.  There are $q^n$ cylinders of depth $n$, each of diameter $r_n$.  Let $\mu$ be the uniform Bernoulli probability measure, so every depth-$n$ cylinder has measure $q^{-n}$.  We claim that
\[
 \Hdim(\Sigma_q,d_\kappa)
 =\liminf_{n\to\infty}\frac{n\log q}{\epsilon\rhok(n)}.
\]
If $s$ is larger than the liminf, choose a subsequence on which the quotient is bounded above by some $s'<s$.  Along that subsequence the depth-$n$ cylinder cover has total $s$--content
\[
 q^n r_n^s
 =\exp\!\bigl(n\log q-s\epsilon\rhok(n)\bigr)
 \le e^{-(s-s')\epsilon\rhok(n)}\longrightarrow0.
\]
This proves the upper bound.

If $s$ is smaller than the liminf, then $q^{-n}\le r_n^s$ for all sufficiently large $n$.  Given a sufficiently small radius $r$, choose the unique $n$ with $r_n\le r<r_{n-1}$.  By the closed-ball convention, $B(\xi,r)$ is exactly the depth-$n$ cylinder through $\xi$.  Therefore
\[
 \mu(B(\xi,r))=q^{-n}\le r_n^s\le r^s.
\]
The mass distribution principle gives the lower bound.  Finally, Lemma~\ref{lem:rho-average} evaluates the displayed liminf as $K_\infty\log q/\epsilon$ when $K_\infty<\infty$ and as $\infty$ otherwise.
\end{proof}

%=============================================================================
\section{Lower Assouad dimension}\label{sec:dimension}
%=============================================================================

For background on Assouad-type dimensions, see \cite{Luukkainen98,KLV13}.
The global form of Proposition~\ref{prop:lower-assouad} is contained in
\cite[Lem.~2.1]{KLV13}; we include the proof because our hypothesis is stated
only below a fixed cutoff.

\begin{definition}
For a metric space $(Z,d)$, let $N_r(B(z,R))$ be the least number of radius-$r$ balls needed to cover $B(z,R)$.  The lower Assouad dimension is
\[
\ldimA(Z,d)=\sup\left\{s\ge0:\begin{array}{l}
\text{there is }c>0\text{ such that }N_r(B(z,R))\ge c(R/r)^s\\
\text{for every }z\in Z\text{ and }0<r<R<\diam Z
\end{array}\right\}.
\]
\end{definition}

\begin{proposition}[Local uniform perfectness and lower dimension]\label{prop:lower-assouad}
Let $(Z,d)$ be a bounded metric space with at least two points.  Suppose that Definition~\ref{def:UP} holds with constants $S>1$ and $r_0>0$.  Then
\[
 \ldimA(Z,d)\ge\frac{\log2}{\log(2S+1)}.
\]
\end{proposition}

\begin{proof}
Put $A=2S+1$ and $s=\log2/\log A$.  If $0<R\le r_0$ and $z\in Z$, apply Definition~\ref{def:UP} at the radius $2SR/A$ to find $w$ such that
\[
 \frac{2R}{A}<d(z,w)\le\frac{2SR}{A}.
\]
The two balls $B(z,R/A)$ and $B(w,R/A)$ are contained in $B(z,R)$, and their centers are more than $2R/A$ apart.  Repeat the same construction in every child ball.  At a transition from level $j-1$ to level $j$, a new center moves by at most $2SR/A^j$.  If two level-$k$ centers first separate at level $j$, their level-$j$ ancestors are more than $2R/A^j$ apart, whereas the total subsequent displacement of each branch is at most
\[
 \sum_{i=j+1}^{k}\frac{2SR}{A^i}
 =\frac{R}{A^j}\bigl(1-A^{-(k-j)}\bigr),
\]
because $A-1=2S$.  The two level-$k$ centers are therefore more than $2R/A^k$ apart.  Thus the iteration yields $2^k$ such centers in $B(z,R)$.  Hence, whenever
\[
 \frac{R}{A^{k+1}}<r\le\frac{R}{A^k},
\]
a radius-$r$ ball contains at most one of these centers, and therefore
\[
 N_r(B(z,R))\ge 2^k
 \ge \frac12\left(\frac{R}{r}\right)^s.
\]

For arbitrary $0<R<\diam Z$, use the preceding construction at
\[
 R'=\min\{R,r_0\}
\]
inside $B(z,R)$.  If $r<R'$, the preceding estimate and $R\le\diam Z$ give
\[
 N_r(B(z,R))
 \ge\frac12\left(\frac{R'}{r}\right)^s
 \ge\frac12\min\left\{1,\left(\frac{r_0}{\diam Z}\right)^s\right\}
       \left(\frac{R}{r}\right)^s.
\]
If $r\ge R'$, then necessarily $R>r_0$ and hence $R'=r_0$.  In this case $N_r(B(z,R))\ge1$ and
\[
 \frac{R}{r}\le\frac{\diam Z}{r_0}.
\]
After decreasing the multiplicative constant once more, the same lower estimate follows.  Thus $s$ is admissible in the definition of the lower Assouad dimension.
\end{proof}

\begin{corollary}\label{cor:tree-lower-assouad}
For every $q\ge2$ and every increasing concave sublinear $\kappa$,
\[
 \ldimA(\Sigma_q,d_\kappa)
 \ge\frac{\log2}{\log(2e^\epsilon+1)}>0.
\]
If $\kappa$ is unbounded, $(\Sigma_q,d_\kappa)$ is non-doubling and has infinite Hausdorff dimension.
\end{corollary}

\begin{proof}
Use Theorem~\ref{thm:regular-dichotomy}(i), Proposition~\ref{prop:lower-assouad}, and Theorem~\ref{thm:regular-dichotomy}(iii)--(iv).
\end{proof}

\subsection*{Acknowledgments}
The work of S.~Han and Q.~Liu motivated the questions considered here.
The author thanks G.~Pallier and Y.~Cornulier for discussions on sublinear
boundary structures. The author is grateful to an anonymous referee for a
careful report on an earlier version of this manuscript. The referee's comments
prompted a substantial revision and helped sharpen the scope, statements, and
exposition of the present paper. This work was supported by the National
Research Foundation of Korea (NRF), funded by the Korean government (MSIT),
grant RS-2025-00513595.

%=============================================================================
% References
%=============================================================================

\end{document}